\documentclass[11pt,a4paper]{article}
\usepackage{amsmath}
\usepackage{amsthm}
\usepackage{amssymb}
\usepackage{amsfonts}
\usepackage{setspace}  
\usepackage{anysize}

\usepackage{pstricks,pst-node,pst-text,pst-3d}
\usepackage{amscd}        
\usepackage{graphics}
\usepackage{color}
\usepackage[all]{xy}
\usepackage{amsthm}   
\usepackage{natbib}
\usepackage{mathrsfs}
\usepackage[all]{xy}
\usepackage{textcomp}
\numberwithin{equation}{section}
\theoremstyle{plain}
\newtheorem{theorem}{Theorem}[section]
\newtheorem{lemma}[theorem]{Lemma}
\newtheorem{corollary}[theorem]{Corollary}
\newtheorem{proposition}[theorem]{Proposition}
\theoremstyle{definition}
\newtheorem{definition}[theorem]{Definition}
\newtheorem{example}[theorem]{Example}
\theoremstyle{remark}

\def\os{{\cal O}}
\def\ra{\rightarrow}

\def\Z{\mathbb Z}

\def\N{\mathbb N}
\def\qco{\mathfrak Qco}
\def\mfR{\mathbf R }
\def\mfM{\mathbf M}

\newcommand{\Ext}{\operatorname{Ext}}
\newcommand{\id}{\operatorname{id}}

\renewcommand{\theenumi}{\roman{enumi}}
\renewcommand{\labelenumi}{(\theenumi)}
\begin{document}
\pagestyle{myheadings}

\enlargethispage{\baselineskip}
\title{\vspace{-2cm} A Lazard-like theorem for quasi-coherent sheaves
\thanks{2010 {\it Mathematics Subject
Classification}.14F05,18F20,55R25,18A30,46M20.}
\thanks {{\it Keywords}. quasi--coherent sheaves, flat Mittag Leffler, almost
projective, vector bundle}}

\author{Sergio Estrada\footnote{Departamento de Matem\'{a}tica Aplicada,
Universidad de Murcia, Campus del Espinardo, 30100 Murcia, Spain.
e-mail: \texttt{sestrada@um.es}.}
 \hspace{0.3em}  Pedro A. Guil Asensio \footnote{Departamento de
Matem\'{a}ticas,
Universidad de Murcia, Campus del Espinardo, 30100 Murcia, Spain.
e-mail:  \texttt{paguil@um.es}} \hspace{0.3em}
Sinem Odaba\c{s}{\i}\footnote{Departamento de Matem\'{a}ticas,
Universidad de Murcia, Campus del Espinardo, 30100 Murcia, Spain.
e-mail:  \texttt{sinem.odabasi1@um.es}.}}

\date{}
 \maketitle
\renewcommand{\theenumi}{\arabic{enumi}}
\renewcommand{\labelenumi}{\emph{(\theenumi)}}

\begin{abstract}

We study filtration of quasi--coherent sheaves. We prove a version of
Kaplansky Theorem for quasi--coherent sheaves, by using Drinfeld's notion of
almost projective module and the Hill Lemma. We also show a Lazard-like
theorem for flat quasi--coherent sheaves for quasi--compact and semi--separated
schemes which
satisfy the resolution property.

\vspace{0.5cm}

\end{abstract}

\section{Introduction}

The interlacing between quasi--coherent sheaves on an arbitrary scheme $X$ and
certain compatible system of modules was outlined in \cite{relativehom}.

This paper is devoted to exploit this relation in case $X$ is quasi--compact
and semi--separated, and in particular when $X$ is a closed
subscheme of $\mathbb{P}^n(R)$.
Namely, recently Drinfeld in \cite{drinfeld} has proposed new notions of
infinite dimensional vector bundles by using flat Mittag-Leffler modules,
projective modules and almost projective modules. In this paper we focus on the
case of almost projective modules. Following Drinfeld, an $R$-module $M$ is
almost projective whenever is a direct summand of the coproduct of a projective
$R$-module and a finitely generated one.
It is then clear that every projective module is in turn almost projective.
Then a quasi--coherent sheaf $\mathcal M$ on $X$ is locally almost projective if
for every open affine subset $\mathrm{Spec}R\subseteq X$ the $R$--module of
sections $\Gamma(\mathrm{Spec} R,\mathcal{M})$ 
is almost projective. We will show that every
locally almost projective quasi--coherent sheaf on $X$ can be {\it filtered} by
locally countably generated almost projective quasi--coherent sheaves. The
precise formulation of our result is as follows:

\medskip\par\noindent
\textbf{Theorem A}.
{\it Let $\mathcal{M}$ be a quasi--coherent sheaf on $X$. Then, there exists a
continuous chain of quasi--coherent subsheaves $(\mathcal{M}_{\alpha}:\
\alpha\leq\lambda)$ of $\mathcal{M}$ such that:
\begin{itemize}
 \item $\mathcal{M}=\bigcup_{\alpha<\lambda}\mathcal{M}_{\alpha}$.
\item $\mathcal{M}_{\alpha+1}/\mathcal{M}_{\alpha}$ ($\alpha<\lambda$) is
locally
countably generated almost projective.
\end{itemize} }

\medskip\par
In the particular case that $\Gamma(\mathrm{Spec} R,\mathcal{M})$ 
is in fact projective, we get Drinfeld's notion of (infinite dimensional)
vector bundle (see \cite[Section 2, Definition]{drinfeld}) and hence Theorem A
specializes to get that $\mathcal{M}_{\alpha+1}/\mathcal{M}_{\alpha}$
($\alpha<\lambda$) is a locally
countably generated vector bundle. Kaplansky Theorem states that every
projective $R$-module $P$ can be written as
$P=\oplus_{\alpha<\kappa}P_{\alpha}$, with $P_{\alpha}$ countably generated and
projective. This is equivalent to saying that every projective module can be
filtered by countably generated and projective modules. In the quasi--coherent
situation this equivalence is no longer true and in fact it seems unlikely to
get such a direct sum decomposition for vector bundles. But our Theorem A
precisely states that Kaplansky Theorem still holds in $\qco(X)$ when replacing
direct sum decomposition by filtration. Our proof is based on the fact that a
Kaplansky-like theorem is also true for almost projective $R$-modules
(Proposition \ref{ap}).

But perhaps the most interesting application of the techniques provided in this
paper is in Theorem B:
\medskip\par\noindent
\textbf{Theorem B}.
{\it Let $X$ be a quasi--compact and semi--separated scheme having enough
locally countably
generated
vector bundles (for instance if $X$ is noetherian, separated, integral and
locally factorial).
Let $\mathcal{F}$ be a flat quasi-coherent sheaf on $X$. Then
$\mathcal{F}=\varinjlim \mathcal{F}_i$, where
$\mathcal{F}_i$ is locally countably generated and flat with ${\mathcal
V}\mathrm{dim}\,{\mathcal F}_i\leq 1$ (where $\mathcal V$ is the class of all
vector bundles on $X$).}

\medskip\par
Lazard in \cite[Theorem 1.2]{lazard} showed that every left $R$-module is a
directed colimit of
finitely generated free modules, so in particular a directed colimit of flat
and finitely presented modules. Theorem B shows that every quasi--coherent sheaf
is a directed colimit of {\it locally countably generated} flats. 
It is well-known that a finitely related flat module is projective.
This seems not to be true for quasi--coherent sheaves, so we replace in Theorem
B this condition by saying that the quasi--coherent sheaves $\mathcal{F}_i$ are
{\it locally of finite projective dimension} $\leq 1$, or, in other words, the
dimension of $\mathcal{F}_i$ with respect to the class of all vector bundles is
1 at most.
Finally, in case $X\subseteq \mathbb{P}^n(R)$ is a closed subscheme (and $R$ is
commutative noetherian, or just commutative if $n=1$) we point up in Corollary
\ref{cons} that we can find a common upper bound of the projective dimensions of
the $\mathcal{F}_i$'s.

The paper is structured as follows: in Section \ref{prem} we will give all
notions and properties of the classes of modules that we will use in the sequel
(almost projective and flat Mittag-Leffler modules) as well as, the notion of
filtration with respect to a class $\mathcal L$ in a Grothendieck category. We
also give the statement of the Hill Lemma in this section. Section \ref{summar}
is devoted to summarize from \cite{relativehom} the equivalence between
$\qco(X)$ and certain category of representations by modules of a quiver. Then
we will make a explicit construction of this equivalence in case $X\subseteq
\mathbb{P}^n(R)$ is a closed subscheme. Finally Section \ref{filtr} contains
the main results of the paper and, in particular, the proofs of Theorems A and
B.
\section{Preliminaries}\label{prem}

Throughout the paper all rings considered are commutative.
Let us recall from \cite[Section 4]{drinfeld} the definition of an almost
projective module. As we will see, this notion generalizes the
notion of  a projective module.

\begin{definition}
Let $R$ be a ring. An \emph{elementary almost
projective $R$-module} is an $R$-module isomorphic  to a direct sum
of a projective $R$-module and a finitely generated one. An
\emph{almost projective $R$-module} is a direct summand of an
elementary almost projective module.
\end{definition}
\begin{proposition}\label{ap}
Every almost projective module is a direct sum of countably
generated almost projective modules.
\end{proposition}
\begin{proof}
Let $T$ be an almost projective $R$-module. Then there exists a
projective  $R$-module $P$ and  a finitely generated $R$-module $M$
such that $T$ is a direct summand of $P\oplus M$. By Kaplansky's
theorem, we know that $P$ is a direct sum of countably generated
projective $R$-modules, say  $P=\bigoplus_{i\in I}P_i$. Then, there
exists an $R$-module $K$ such that
$$\left(\bigoplus_{i\in I}P_i\right)\oplus M =T \oplus K.$$
Since $T$ is a direct summand of a direct sum of countably generated
modules, $T$ is again a direct sum of countably generated modules by
~\citet[Theorem 26.1]{fuller}. Say $T=\bigoplus_{j\in J}T_j$ for some
index set $J$ where $T_j$ is a countably generated module for every
$j\in J$. Clearly, each $T_j$ is a direct summand of $P\oplus M$.
This implies that $T_j$ is an almost projective $R$-module for each
$j\in J$ and $T$ is a direct sum of countably generated almost
projective modules.
\end{proof}

\begin{definition}
Let $R$ be a ring and $M$ be a right $R$-module. $M$ is a
\emph{Mittag-Leffler} module if the canonical map $M\otimes_R
\prod_{i\in I}M_i \rightarrow \prod_{i\in I}M \otimes_R M_i$ is
monic for each family of left $R$-modules $(M_i | \ i\in I)$.
\end{definition}

\begin{theorem}(\cite{RG})\label{ML}
The following conditions are equivalent:
\begin{enumerate}
\item $M$ is flat Mittag-Leffler $R$-module.
\item Every finite or countable subset of $M$ is contained in a
countably generated projective submodule $P \subseteq M$ such that
$M/P$ is flat.
\end{enumerate}
\end{theorem}
Notice that the second condition of the previous theorem allows to write
every flat Mittag-Leffler $R$-module $M$ as a direct union
of projective and countably generated submodules.

Let $\mathcal{C}$ be a Grothendieck category. A well-ordered direct
system of objects of $\mathcal{C}$, $\mathcal{A}=(A_\alpha| \ \alpha
\leq \mu)$, is said to be \emph{continuous} if $A_0=0$ and
$A_\alpha=\varinjlim_{\beta < \alpha} A_\beta$ for all limit
ordinals $\alpha \leq \mu$. If all  morphisms in the system, $f_{
\alpha\beta}$, are monomorphisms then the sequence $\mathcal{A}$ is
called a \emph{continuous chain of modules}.

Let $\mathcal{C}'$ be a class of  objects of $\mathcal{C}$. An object $M\in
\mathcal{C}$ is said to be \emph{$\mathcal{C}'$-filtered} if there is a
continuous chain $\mathcal{A}=(A_\alpha| \ \alpha
\leq \mu)$ of subobjects of $M$ with $M=A_\mu$ and each of the objects
$A_{\alpha+1}/A_\alpha$ is isomorphic to an
object of $\mathcal{C}'$, where $\alpha < \mu$. The chain
$(A_\alpha|\ \alpha \leq \mu)$ is called a
\emph{$\mathcal{C}'$-filtration} of $M$.

The following lemma, known as Hill Lemma, helps us to expand a single
filtration of a module $M$ to a complete lattice of its submodules having  
some good properties. It is one of the most important tools to prove our main
results.

\begin{lemma}[Hill Lemma]\cite[Theorem 4.2.6]{triflaj}\label{hill}
Let $R$ be a ring, $\kappa$ an infinite regular cardinal and
$\mathcal{C}$
 a set of $< \kappa$-presented modules. Let  $M$ be a module with a 
$\mathcal{C}$-filtration $\mathcal{M}=(M_\alpha |\ \alpha \leq \sigma)$ for some
ordinal $\sigma$. Then there is a family $\mathcal{H}$ consisting of submodules
of $M$ such that:
\begin{enumerate}
\item $\mathcal{M} \subseteq \mathcal{H}$.
\item $\mathcal{H}$ is closed under arbitrary sums and intersections 
(that is, $\mathcal{H}$ is a complete sublattice of the lattice of submodules of
$M$).
\item If $N,P \in \mathcal{H}$  such that $N \subseteq P$, 
then there exists a $\mathcal{C}$-filtration $(\overline{P}_\gamma  | \gamma
\leq \tau)$ of the module $\overline{P}=P/N$  $\tau \leq \sigma$ such that  and
for each $\gamma < \tau$, there is a $\beta < \sigma$ with 
$\overline{P}_{\gamma+1}/ \overline{P}_\gamma$ isomorphic to
$M_{\beta+1}/M_\beta $.
\item If $N\in \mathcal{H}$ and $X$ is a subset of $M$ of cardinality $<
\kappa$, then there is a $P\in \mathcal{H}$ such that $N \cup X \subseteq P$
and $P/N$ is $\kappa$-presented.
\end{enumerate}
\end{lemma}

\section{$\qco(X)$ as a Category of Representations}\label{summar}
Let $X$ be a scheme and $\qco(X)$ be the
category of quasi-coherent
sheaves on  $X$. Following \cite{relativehom}, the aim of this section is to
give an equivalent category to $\qco (X)$ which will allow us to understand
$\qco(X)$ in terms of certain compatible systems of modules. 

A \emph{ quiver} $Q$ is a directed graph which is given by the pair
$(V,E)$, where  $E$  denotes the set of all edges of the quiver $Q$
and $V$ is the set of all vertices. A \emph{representation}
$\mathbf{R}$ of a quiver $Q$ in the category of commutative rings means that
for
each vertex $v\in V$ we have a ring $R (v)$ and a ring homomorphism
$\mfR(a):R(v)\longrightarrow R(w),$ for each edge $a:v\rightarrow
w$.

An \emph{$\mathbf{R}$-module} $\mathbf{M}$ is given by an $R
(v)$-module $M(v)$, for each  vertex $v\in V$, and an $R(v)$-linear
morphism $\mathbf{M}(a):M(v) \longrightarrow M (w)$ for each edge
$a:v \rightarrow w \in E$.  Since $\mfR(a)$ is a ring homomorphism
for an edge $a:v\ra w$, the  $R(w)$-module $M(w)$ can be thought as
an $R(v)$-module.

An $\mathbf{R}$-module $\mathbf{M}$ is \emph{quasi-coherent} if for
each edge $a:v \ra w$, the morphism
$$\id _{R(w)} \otimes_{R(v)} \mfM(a):R(w) \otimes_{R(v)}  M(v) \rightarrow  R(w) \otimes_{R(w)} M(w) $$
is an  $R(w)$-module isomorphism. For a fixed quiver $Q$ and a fixed
representation
$\mfR$ of $Q$, the category of quasi-coherent
$\mfR$-modules  is defined as the full subcategory of the
category $\mfR$-Mod containing all quasi-coherent $\mfR$-modules.
We will denote it by $\mfR_{Qco}$-Mod. We will say that the
\emph{representation $\mfR$ is flat} if the ring $R(w)$ is a flat
$R(v)$-module for each edge $a:v\ra w$. If the representation $\mfR$
is flat, then the category $\mathbf{R}_{Qco}$-Mod is a
Grothendieck category.

Consider the category of quasi-coherent sheaves on a scheme
$(X,\os_X)$, denoted by $\qco(X)$. By the definition of a scheme,
the scheme $X$ has a family $\mathcal B$ of  affine open subsets
which is a base for $X$ such that this family uniquely determines
the scheme $(X,\os_X)$(for example, it is enough to take the family
of  the affine open subsets covering $X$ and $U \cap V$ for all
$U,V$ in this family). And also this family helps to uniquely
determine   the quasi-coherent $\os_X$-modules. That is, a
quasi-coherent $\os_X$-module is determined by giving an
$\os_X(U)$-module $M_U$ for each $U$ and a linear map $f_{UV}:M_U
\rightarrow M_V$ whenever $V \subseteq U$, $V,U \in \mathcal{B}$,
satisfying; 
\begin{enumerate}
\item $\os_X (V) \otimes_{\os_X (U)} M_U \rightarrow 
\os_X (V) \otimes_{\os_X (V)} M_V$ is an isomorphism with respect to the
morphism $\id \otimes f_{UV} $ for all $V \subseteq U$.
\item If $W \subseteq V \subseteq U$, where $W,V,U \in \mathcal{B}$, 
then the composition \mbox{$M_U \rightarrow M_V \rightarrow M_W$} gives $M_U
\rightarrow M_W$.
\end{enumerate}

In this way, we are able to construct a quiver $Q=(V,E)$ with
respect to the scheme $(X,\os_X)$. Let $\mathcal B$ be a base of the
scheme $X$ containing   affine open subsets such that $\os_X$ is
$\mathcal B$-sheaf. Now, define a quiver $Q$ having the family
$\mathcal B$ as the set of  vertices, and an edge between two affine
open subsets $U,V \in \mathcal B$ as the only arrow $U \ra V$
provided that $V \subsetneq U$. Fix this quiver. Take the
representation $\mathbf{R}$ as $R(U)=\os_X(U)$ for each $U \in
\mathcal B$ and the restriction map $\rho_{UV}:\os_X(U) \rightarrow
\os_X(V)$ for the edge $U\rightarrow V$. Then the functor
$$\Phi :\qco (X) \longmapsto \mfR_{Qco}\textrm{-Mod},$$
which was defined by the above argument is well-defined and, in fact, it is
an equivalence of categories. Because of this equivalence we will often identify
a quasi--coherent sheaf $\mathcal{M}$ with its corresponding quasi--coherent
$R$-module $\mathbf{M}$ and vice versa.

\begin{example}\label{projn2}
Let $X=\mathbb{P}_R^n=\mathrm{Proj}\, R[x_0,\ldots,x_n]$, $n\in \N$. Then
take a base
containing the affine open sets $D_+(x_i)$  for all $i=0,\ldots n$,
and all possible intersections. In this case, our base contains  basic 
open subsets of this form
$$D_+( \prod_{i \in v}x_i),$$ where $v \subseteq \{0,1, \ldots ,n \}$.
So, the vertices of our quiver are all subsets of
$\{0,1,\ldots,n\}$ and we have only one edge $v\ra w$ for each
$v\subseteq w \subseteq \{0,1,\ldots,n\}$ since $D_+(\prod_{i \in
w}x_i) \subseteq D_+(\prod_{i \in v}x_i)$. Then the structure sheaf takes the
following values  
$$\os_{\mathbb{P}_R^n}(D(\prod_{i \in v}x_i)) =R[x_0,\ldots,x_n]_{(\prod_{i \in v}x_i)}$$
on each basic open set, and  it is isomorphic to the polynomial ring on
the ring $R$ with the variables $\frac{x_j}{x_i}$ where $j=0, \ldots
,n$ and $i \in v$. We will denote this polynomial ring by $R[v]$.
Then the representation $\mathbf{R}$ with respect to this quiver has
vertex $R(v)=R[v]$ and edges $R[v] \hookrightarrow R[w]$ as long as
$v \subseteq w$.

Finally, an $\mathbf{R}$-module $\mathbf{M}$ is quasi-coherent if
and only if
$$S_{vw}^{-1} f_{vw} :S_{vw}^{-1}M(v) \longrightarrow S_{vw}^{-1}M(w)=M(w)$$
is an isomorphism as $R[w]$-modules for each $f_{vw}:M(v)
\rightarrow M(w)$ where $S_{vw}$ is the multiplicative group
generated by the set $\{x_j/x_i|\textrm{ }  j \in w \setminus v
,\textrm{ }i \in v\} \cup \{1\}$ and $v\subset w$.

\end{example}

A closed subscheme  $X \subseteq \mathbb{P}_R^n$ is given by a
quasi-coherent sheaf of ideals, i.e. we have an ideal $J(v)
\subseteq R[v]$ for each $v$ with
$$R[w]\otimes_{R[v]} J(v) \cong J(w),$$
when $v \subseteq w$. This means $J(v) \rightarrow J(w)$ is the
localization of $J(v)$ by the same multiplicative set $S_{vw}$ as
above. Then
$$\frac{R[v]}{J(v)} \rightarrow \frac{R[w]}{J(w)}$$
 is a localization with respect to the set $\overline{S}_{vw}$. To simplify the
notation, we will use
 $R(v)$ instead of $\frac{R[v]}{J(v)}$ to represent the representation of rings
associated to a closed subscheme $X$ of $\mathbb{P}_R^n$.

\begin{example}
The construction of the previous example can be extended to quasi--compact and
semi--separated schemes.
Let $X$ be a quasi--compact and semi--separated scheme, and let
$\mathcal{U}=\{U_0,\ldots,U_n\}$ be an affine open cover of $X$. Let us
construct a quiver $Q_X$ whose vertices
are the subsets $ v\subseteq  \{0, 1, 2, \ldots , n\}$,  $v\neq \emptyset$, and
where $v$ represents the affine open $\cap_{k\in v}U_k$, and
where there is a unique arrow $v\to w$ when $v\subseteq w$, and corresponds to
the canonical inclusion $\cap_{k\in w}U_k\hookrightarrow \cap_{k\in v}U_k$.
Then a quasi--coherent sheaf $\mathcal M$ on $X$ corresponds to a
quasi--coherent $\mathbf{R}$-module $\mathbf{M}$ on $Q_X$ and vice versa.
\end{example}

\section{Filtration of quasi--coherent sheaves}\label{filtr}

It is known that there is a bijection between the class of
vector bundles in the sense of classical algebraic geometry and the
class of all locally free coherent $\os_X$-modules of finite rank (see
\cite{hartshorne}).
So, in Sheaf Theory, a vector bundle is a
locally free coherent $\os_X$-module of finite rank.
Following \citet{drinfeld}, we can achieve at least three different
generalizations of this definition. The first one is getting just by avoiding
the finitely generated assumption.

\begin{definition}\cite[Section 2]{drinfeld}\label{v.c}
Let $(X,\os_X)$ be a scheme. A quasi-coherent $\os_X$-module
$\mathcal{F}$ is said to be a \emph{vector bundle} (in the sense
\cite{drinfeld})  if $\mathcal{F}(U)$ is a projective
$\os_X(U)$-module for every  affine open subset $U$ of $X$.
\end{definition}

But then, according to \cite[Sections 2 and 4]{drinfeld}, we get other
generalizations of classical vector bundles if in the previous
definition we replace ``projective'' by ``flat Mittag-Leffler'' or by ``almost
projective''. These yield to the notions of {\it locally flat Mittag-Leffler}
and {\it locally almost projective} quasi--coherent sheaf.

So, if $\mfR$ is a representation of the structure sheaf of the scheme
$(X,\os_X)$, a
 vector bundle (resp. a locally flat Mittag-Leffler or a locally almost
projective) $\mathcal M$ corresponds to a
 unique element $\mathbf M$ in $\mfR_{Qco}$-Mod such
 that each $M(u)$ is a projective (resp. flat Mittag--Lefler or almost
projective) $R(u)$-module for every vertex $u$. 

We are interested whether the converse of this property holds. Namely, if
$\mathbf{M}$ is a quasi--coherent $\mathbf{R}$-module over some quiver
$Q=(V,E)$ representing the scheme $(X,\mathcal{O}_X)$, and such that
$M(u)$ is projective (resp. flat Mittag-Leffler or almost projective) for each
vertex $u\in V$, then $\mathbf{M}$ must be a vector bundle (resp. locally flat
Mittag-Leffler or locally almost projective) in the above sense. To this end we
need the following lemma, which is essentially due to Raynaud Gruson and
pointed up by Drinfeld.
\begin{lemma}\label{prev_ML}
 Let $R\to S$ be a ring map and let $M$ be an $R$-module. If $M$ is a
projective (resp. flat Mittag-Leffler or almost projective),
then $S\otimes_R M$ is a projective (resp. flat Mittag-Leffler or
almost
projective) $S$-module. If, in addition, $R\to S$ is faithfully flat then the
converse is also true, that is, $S\otimes_R M$ being projective
(resp. flat Mittag-Leffler or almost
projective) $S$-module implies that $M$ is such.
\end{lemma}
\begin{proof}
If $M$ is a projective (resp. flat Mittag-Leffler or almost projective)
$R$-module, it is straightforward to check that $S\otimes_R M$ is a
projective
(resp. flat Mittag-Leffler or almost projective) $S$-module.
For the second part see \cite[Section 2]{drinfeld} and \cite[Theorem
4.2(i)]{drinfeld}. 

\end{proof}
\begin{proposition}
 Let $X$ be a scheme. Let $\mathcal{F}$ be a quasi-coherent sheaf on $X$. The
following are equivalent:
\begin{enumerate}
\item $\mathcal{F}$ is vector bundle (resp. locally flat Mittag--Leffler or
locally almost projective)  and
\item there exists an affine open covering $X=\cup_{i\in I} U_i$ such that the
$\mathcal{O}_X(U_i)$-module
$\mathcal{F}(U_i)$ is projective (resp. flat Mittag--Lefler or almost
projective) for every $i\in I$.
\end{enumerate}
\end{proposition}
\begin{proof}
$(1)\Rightarrow (2)$ is immediate.

\noindent
$(2)\Rightarrow (1)$. Let $U\subseteq X$ be an arbitrary open affine. And write
$U=\bigcup_{j=1}^n D(f_j)$ where, for each $j$, $D(f_j)$ is a basic open of
some $U_i$, $i\in I$. Let us denote $M=\mathcal{F}(U)$ and
$R=\mathcal{O}_X(U)$. Therefore $\mathcal{O}_X(D(f_j))=R_{f_j}$. Then the
hypothesis and Lemma \ref{prev_ML} applied to the ring map
$\mathcal{O}_X(U_i)\to \mathcal{O}_X(D(f_j))$, give that
$\mathcal{F}(D(f_j))=R_{f_j}\otimes_R M=M_{f_j}$ is a projective (resp. flat
Mittag--Leffler or
almost projective) $R_{f_j}$-module. Let
$S=\prod_{j=1}^mR_{f_j}$. Now $R\to S$
is faithfully flat and $S\otimes_R M=\prod_{j=1}^n M_{f_j}$ is a projective
(resp. flat Mittag--Lefler and almost projective) $S$-module. Hence we infer,
again by Lemma \ref{prev_ML}, that
$M$ is projective (resp. flat Mittag-Leffler or almost projective) $R$-module.
\end{proof}

From now on we will assume that the scheme $X$ is quasi--compact and
semi--separated. Soon it will
become clear that the only requirement needed in our results is that $X$ can be
covered by at most countably many affine opens. But just for a sake of
simplicity the reader may assume that $X\subseteq \mathbb P_R^n$ is a closed
subscheme. We point out that we will write down explicitly when we need to
impose further assumptions on the scheme. 

By the previous comments, we will represent each vector bundle (resp.
locally flat Mittag--Leffler or locally almost projective)
$\mathcal P$ on $X$ by a quasi--coherent $\mathbf{R}$-module
$\mathbf P$ over some quiver $Q=(V,E)$ representing $X$, such that $ P(v)$ is a
projective (resp. flat Mittag--Leffler or
almost projective) $R(v)$-module for each
$v \in V$. Our first aim in this section is to prove Theorem
\ref{ap2} (Theorem A in the introduction). As a consequence of it we provide in
Corollary \ref{sonuc} with a version of Kaplansky Theorem for vector bundles in
the Drinfeld's sense.

The following  lemma and proposition, which are modified versions
of \citet[Lemma 3.2, Proposition 3.3]{relativehom}, have importance
in proving our main results.

\begin{lemma}\label{q.c}
Let $\mfR ' \equiv R(v) \to R(w)$ be a part of the ring
representation $\mfR$ of $X$
where $w \subseteq v $. Suppose that we have
a quasi-coherent $\mfR'$-module
$$ M(v)\stackrel{f} {\longrightarrow} M(w) $$
and   two countable subsets $X(v)$ and $ X(w) $ of $M(v)$ and
$M(w)$, respectively. Then there exists  a quasi-coherent $\mathbf R
'$-submodule
$$M'(v)  \longrightarrow M'(w) $$
of $ M(v)\stackrel{f} {\longrightarrow} M(w) $ such that $X(v)
\subseteq M'(v) \subseteq M(v)$, $X(w) \subseteq M'(w)  \subseteq
M(w)$ and $M'(v),M'(w)$ are countably generated modules over $R(v)$
and $R(w)$, respectively.
\end{lemma}
\begin{proof}

Let $t \in X(w)$. Then, because of the quasi-coherence, there exists
$Y_t=\{y_1,\cdots,y_{k_t}\}\subseteq M(v)$ such that
$$t=\sum_{i=1}^{k_t}r_i f(y_i),\ \ r_i\in R(w).$$
Take the submodule $M'(v)$ of $M(v)$ generated by $X(v) \cup Y$
where $Y$ consists of all of  $Y_t$  which has been found for each
$t \in X(w)$ as above . Since $|X(v) \cup Y|\leq \aleph _0 +
\aleph_0=\aleph _0$, $M'(v) $ is countably generated. Let $M'(w)$ be
the $R(w)$-submodule  of $M(w)$ generated by $f(M'(v))$. Clearly
$M'(w)$ is a countably generated submodule of $M(w)$ containing
$X(w)$. Let us see that the submodule
$$M'(v) \xrightarrow{f|_{M'(v)}} M'(w)$$
 is
quasi-coherent. Since the morphism $\varphi=s\circ (f\otimes_{R(v)}id_{R(w)})$
is an
isomorphism (where $s(r_w\otimes m_w)=r_wm_w$, $r_w\in R(w)$ and $m_w\in M(w)$),
we only need to show
that $\varphi(M'(v)\otimes_{R(v)} R(w))$ is  the
$R(w)$-module $M'(w)=R(w)f(M'(v))$. Indeed,
$\varphi (M'(v)\otimes_{R(v)} R(w))$ is equal to the
$R(w)$-module generated by $f(M'(v))$, that is, to the $R(w)$-module
$M'(w)$. This implies that $M'(v) \rightarrow M'(w)$ is a
quasi-coherent $\mfR'$-submodule of $M(v) \ra M(w)$.
\end{proof}
\begin{proposition}\label{anaprp}   
Let $\mathbf M$ be a quasi-coherent sheaf on 
 $X$. If $X(v) \subseteq  M(v)$ is a countable subset
for each $v
 \subseteq \{ 0,1, \ldots ,n \}$, then there exists a
quasi-coherent submodule $\mathbf M'$ of $\mathbf M$ such that $X(v)
\subseteq M'(v)  \subseteq M(v)$ and $M'(v) $ is a countably
generated $R(v)$-module for all $v \subseteq \{ 0,1, \ldots ,n \}$.
\end{proposition}
\begin{proof}
Let $E=\{ e_i : 0 \leq l \leq k \}$ be the set of  all arrows
defining the quiver of $X$
for some natural number $k$. We will construct by induction a family
of $\mfR$-submodules $\mathbf M^{(m)}$ of $\mathbf M$ satisfying:
\begin{enumerate}
\item $X(v) \subseteq M^{(m)} (v) \subseteq M(v)$ is countably
generated for each $v$ and $m,l \in \N$.
\item \mbox{$M^{(m)} (v) \rightarrow
M^{(m)}(w)$} satisfies the quasi-coherent condition on the edge $l$, whenever 
$m\equiv l \pmod{(k+1)}$ for $m \in \N$ and $l \in E$,
\item $\mathbf M^{(m)} \subseteq \mathbf M^{(m+1)}$ for all $m \in \N$.
\end{enumerate}

When $m \geq n+1$, think of $e_m$ as $e_l$, where $m \equiv l
\pmod{(k+1)}$ and $l\in E$. Let us consider the edge $e_0:v
\rightarrow w$. By applying Lemma \ref{q.c} to this edge, we obtain
$T_0 ^{(0)} (v ) \rightarrow T_0 ^{(0)} (w)$ satisfying the
quasi-coherent condition. And say $T_0 ^{(0)} (u) := X(u)$ for all
$u \subseteq \{0,\ldots,n\}$ different from $v$ and $w$. Now from
$\mathbf T_0 ^{(0)}$, by taking $T_1 ^{(0)} (u)$ as the
$R(u)$-module generated by the sets $\{ T_0 ^{(0)} (u), f_{u',u}(T_0
^{(0)} )|f_{u',u}: M(u') \rightarrow M(u)\}$ where each morphism
$f_{u',u}$ denotes the morphism $\mathbf M(a)$ where  $a:u' \ra u$
$(u,u' \subseteq \{0,\ldots, n \})$, we obtain a locally countably
generated  $\mfR$-submodule $\mathbf T_1 ^{(0)}$. But it is possible
that we may have lost the quasi-coherent condition on $e_0 :v
\rightarrow w$. So, we again apply Lemma \ref{q.c}  to obtain
$\mathbf T_2 ^{(0)}$ such that $T_2 ^{(0)} (v ) \rightarrow T_2
^{(0)} (w)$ satisfies the quasi-coherent condition. And by the same
argument above, we can construct an $\mfR$-submodule $\mathbf T_3
^{^(0)}$. Continuing in this way, we obtain the family $\{ \mathbf
T_n^{(0)}\}_{n\in \N}$.

Define the first term $\mathbf M^{(0)}$ as the direct union of this
family on $n\in \N$. Now assume we have constructed $\mathbf
M^{(m)}$ for $m\in \N$. Let us define $\mathbf M ^{(m+1)}$. Take the
edge $e_{m+1} : v \rightarrow w$. We apply Lemma \ref{q.c} to
$M^{(m)}(v) \rightarrow M^{(m)}(w)$ to obtain $T_0 ^{(m+1)}(v)
\rightarrow T_0 ^{(m+1)}(w)$ which satisfies the quasi-coherent
condition. Define $T_0 ^{(m+1)}(u):=M^{(m)}(u)$ for every $u \neq
v,w$. From this, we can construct  an $\mfR$-submodule $\mathbf
T_1^{(m+1)}$ of $\mathbf M$ by the same method we did before. Again
applying Lemma \ref{q.c} to $T_1 ^{(m+1)}(v) \rightarrow T_1
^{(m+1)}(w)$, we find $\mathbf T_2^{(m+1)}$ such that $T_2
^{(m+1)}(v) \rightarrow T_2 ^{(m+1)}(w)$ is quasi-coherent. By
proceeding in the same way, we obtain the family $\{ \mathbf T_n
^{(m+1)}\}_{n \in \N} $. So, define $\mathbf M^{(m+1)} := \bigcup
_{n \geq 0} \mathbf T_n ^{(m+1)}$. So we have constructed
inductively  the desired family $\{\mathbf M^{(m)}\}_{m \in
\mathbb{N}}$.

Finally, if we let $M'(v):= \bigcup _{m \in \mathbb{N}} M^{(m)}(v)$
for all $v \subseteq \{0,1,2, \ldots ,n \}$, we see that the
properties of being an  $\mfR$-module and the quasi-coherence
condition on each edge are cofinal. So, it follows that $\mathbf M'$
is a quasi-coherent $\mfR$-submodule of $\mathbf M$ containing
$X(v)$ for all $v \subseteq \{0,1,2, \ldots ,n \}$. Clearly, $\mathbf M'$ is
locally countably generated, since
$$|M'(v)|=|\bigcup _{n \in \mathbb{N}} M^{(m)}(v)|$$
for each $v$ and countable union of countable sets is again
countable.
\end{proof}
For the proof of the next Theorem, we need to fix the following notation:
Let $\mathcal{S}_v$ be the class of all countably generated almost
projective $R(v)$-modules for each $v \subseteq \{0,\ldots,n\}$,
$\mathcal{L}$ be the class of all locally countably generated almost
projective quasi-coherent $\mfR$-modules on
$Q_X$ and $\mathcal{C}$ be the class of all
locally almost projective quasi-coherent $\mfR$-modules. Then the
class $\mathcal L$ contains quasi-coherent $\mfR$-modules $\mathbf
M$ such that $M(v) \in \mathcal S_v$ for each edge $v \subseteq
\{0,\ldots,n\}$.

\begin{theorem}\label{ap2}
Every locally almost projective quasi-coherent $\mfR$-module is filtered by
locally countably generated almost projective quasi--coherent $\mfR$-modules.
\end{theorem}

\begin{proof}
Let $\mathbf T$ be a quasi-coherent $\mfR$-module belonging to the
class $\mathcal C$.  By Proposition \ref{ap}, we know that each
$T(v)$ has an $\mathcal{S}_v$-filtration $\mathcal{M}_v$ for all $v
\subseteq \{0,1,2, \ldots ,n \}$. Let $\mathcal{H}_v$ be  the family
associated to $\mathcal{M}_v$ by Lemma \ref{hill} and
$\{m_{v,\alpha} | \alpha < \tau _v \}$ be an $R(v)$-generating set
of the $R(v)$-module $M(v)$. Without lost of generality, we can
assume that for some ordinal $\tau$, $\tau = \tau _v$ for all $v$.

We will construct an $\mathcal{L}$-filtration $(\mathbf M_\alpha |\
\alpha \leq \tau)$ for $\mathbf T$ by induction on $\alpha$. Let
$\mathbf M_0=0$. Assume that $\mathbf M_\alpha$ is defined for some
$\alpha < \tau$ such  that $M_\alpha(v) \in \mathcal{H}_v$ and
$m_{v,\beta} \in M(v)$ for all $\beta < \alpha$ and all $v \subseteq
\{0,1,2, \ldots ,n \}$. Set  $N_{v,0}= M_\alpha (v)$. By
Lemma \ref{hill}(iv), there is a module $N_{v,1} \in \mathcal{H}_v$ such that
$N_{v,0} \subseteq N_{v,1}$ and $N_{v,1} / N_{v,0}$ is countably
generated.

By Proposition \ref{anaprp} (with $\mathbf M$ replaced by $\mathbf T
/ \mathbf M_\alpha $, and $X(v)=N_{v,1}/M_\alpha(v)$) there is a
quasi-coherent $\mathbf R$-submodule $\mathbf T_1 $ of $\mathbf T$
such that $\mathbf M_\alpha \subseteq \mathbf T_1 $ and $\mathbf T_1
/ \mathbf M_\alpha$ is locally countably generated. Then $T_1 (v) =
N_{v,1}+ \langle T_v \rangle$ for a countably subset  $T_v \subseteq
T_1  (v)$, for each $v$. Again by help of Lemma \ref{hill}(iv), there is
a module $N_{v,2} \in \mathcal{H}_v$ such that $T_1  (v)= N_{v,1} +
\langle T_v \rangle \subseteq N_{v,2}$ and $N_{v,2} / N_{v,1}$ is
countably generated.

Proceeding similarly, we obtain a countable chain $(\mathbf T_n  |\
n < \aleph_0)$ of quasi-coherent $\mathbf R$-submodule of $\mathbf
T$, as well as a countable chain $(N_{v,n} | n < \aleph_0)$ of
$R(v)$-submodules of $T(v)$, for each $v$. Let $\mathbf
M_{\alpha+1}= \bigcup _{n < \aleph_0} \mathbf T_n $. Then  $\mathbf
M_{\alpha+1}$ is a quasi-coherent subsheaf of $\mathbf T$ satisfying
$M_{\alpha+1}(v)= \bigcup _{n < \aleph_0} T_n '(v)$ for each $v$. By
Lemma \ref{hill}(ii), we deduce that $M_{\alpha+1}(v) \in \mathcal{H}_v$
and $M_{\alpha+1}(v) / M_\alpha (v)$ is a countably generated almost
projective $R(v)$-module. Therefore $\mathbf M_{\alpha+1} / \mathbf
M_\alpha \in \mathcal{L}$.

Assume $\mathbf M_\beta$ has been defined for all  $\beta < \alpha$
where $\alpha$ is a limit ordinal $\leq \tau$. Then  we define
$\mathbf M_\alpha := \bigcup_{\beta < \alpha} \mathbf M_\beta$.

Since $m_{v,\alpha} \in M_{\alpha+1}(v)$ for all $v$ and $\alpha <
\tau$, we have  $M_\tau(v)=M(v)$. So $(\mathbf M_\alpha |\ \alpha
\leq \tau)$ is an $\mathcal{L}$-filtration of $\mathbf T$.
\end{proof}
Now, as an application of Theorem \ref{ap2}, we can get a version of Kaplansky
Theorem for quasi--coherent sheaves on $X$ (cf. \cite[Corollary
3.12]{modelcategory}).
To do this we just have to restrict $\mathcal S_v$ to the class of all
countably generated
projective $R(v)$-modules for each $v \subseteq \{0,\ldots,n\}$ in the proof of
Theorem \ref{ap2}, and then using Kaplansky Theorem instead of Proposition
\ref{ap}. Notice that in this case,
$\mathcal L$ will be the class of all locally countably generated vector
bundles on $X$ and $\mathcal
C$ be the class of all vector bundles.

\begin{corollary}\label{sonuc}
Every vector bundle on $X$ is a filtration of
locally countably generated vector bundles.
\end{corollary}



Now we prove that locally flat Mittag-Leffler quasi-coherent sheaves are
direct unions of locally countably genererated vector bundles.
\begin{theorem}\label{mlqc}
Every locally flat Mittag-Leffler quasi-coherent sheaf on $X$ is a direct union
of locally countably generated vector bundles.
\end{theorem}
\begin{proof}
Let $\mathbf M$ be a locally flat Mittag-Leffler quasi-coherent
$\mfR$-module. By Theorem \ref{ML}, each
$R(v)$-module $M(v)$ is a union of countably generated projective
submodules $v \subseteq \{0,1,2, \ldots ,n \}$, that is,
$M(v)=\bigcup_{i \in I_v} P^i_v$. W.l.o.g., wee can assume that
$I=I_v$ for each $v \subseteq \{0,1,\ldots ,n\}$.

Let $i\in I$. Set $T^0(v) := P_v^i$. By Proposition \ref{anaprp},
there is a quasi-coherent $\mathbf R$-submodule $\mathbf T^1 $ of
$\mathbf M$ such that $T^0(v) \subseteq  T^1(v) $ and $\mathbf T^1$
is locally countably generated. By Theorem \ref{ML}, there is a
countably generated projective submodule $T^2(v)$ of $M(v)$
containing $T^1(v)$, for each $v$. Again, by Proposition
\ref{anaprp}, there is a quasi-coherent $\mathbf R$-submodule
$\mathbf T^3 $ of $\mathbf M$ such that $T^2(v) \subseteq  T^3(v) $
and $\mathbf T^3$ is locally countably generated. Continuing in the
same way, we obtain a countable chain $(\mathbf T^{n+1})_{n\in \N}$
of locally countably generated quasi-coherent subsheaves of $\mathbf
M$ as well as a countable chain $(T^{2n}(v))_{n\in \N}$ of
$R(v)$-submodules of $M(v)$ contained in $T^{n+1}(v)$, for each $v$.
Let $\mathbf M_i := \bigcup_{n\in \N} \mathbf{T^{n+1}}$. Then
$\mathbf M_i$ is a locally countably generated vector bundles since
$M_i(v)=\bigcup_{n\in \N} T^{2n}(v)$ for each $v$.

Finally, $\mathbf M= \bigcup_{i\in I} \mathbf M_i$ since each
$M_i(v)$ contains $P_v^i$ for all $i\in I$ and $v \subseteq
\{0,1,\ldots ,n\}$.

\end{proof}

Now, let $F$ be a flat $R$-module. Then there exists a short exact
sequence
$$0\rightarrow M \hookrightarrow \bigoplus_J R \rightarrow F
\rightarrow 0.$$ Since $F$ is flat, $M$ is a pure submodule of
$\bigoplus_J R$. Since pure submodules of flat Mittag--Leffler are flat
Mittag--Leffler, we follow that $M$ is flat Mittag--Leffler. Therefore
$M= \bigcup_{i\in I}M_i$ where $M_i$ is a countably generated
projective submodule of $M$ for each $i\in I$. For $J'$ a countable
subset of $J$, set $M_{i,J'}:=M_i$. Then we may find some countable
subset $J''$ of $J$ containing $J'$ and $M_{i,J'}$ is submodule of
$\bigoplus_{J''}R$. If we denote $R_{i,J'}:= \bigoplus_{J''}R$, we get
a commutative diagram
$$\xymatrix{0 \ar[r]& M_{i,J'} \ar@{^{(}->}[r]\ar@{^{(}->}[d] &
R_{i,J'}\ar@{^{(}->}[d] \ar[r] & F_{i,J'}\ar[r]\ar[d]& 0\\
0\ar[r]& M\ar@{^{(}->}[r]&\bigoplus_J R\ar[r] & F \ar[r] & 0 }.$$
Here, $F_{i,J'}$ is countably generated and flat since $M_{i,J'}$ is
pure in  $R_{i,J'}$. If we take their direct limits over $I$ and
countable subsets ${J'}$ of $J$, we get $F=\varinjlim F_{i,J'}$.

\bigskip\par
Now we shall prove the main result of our paper. To do so, we will need to
assume that our scheme $X$ possesses a family of locally countably
generated vector bundles. This is the case whenever $X$
satisfies the resolution property (that is, every coherent sheaf is a quotient
of some finite dimensional vector bundle) because, in that situation, every
quasi--coherent sheaf on $X$ is the filtered union of coherent subsheafs. So the
vector bundles constitute a family of generators of $\qco(X)$. We find examples
of such schemes whenever $X$ is noetherian, separated, integral and locally
factorial by a result of Kleiman (see \cite[Ex. III.6.8]{hartshorne}).

Let us denote by $\mathcal V$ the class of all vector bundles on $X$.
Given a ${\mathbf M}\in \mathbf{R}_{Qco}\textrm{-Mod}$ we say that
${\mathcal V}\mathrm{dim}\,{\mathbf M}\leq n$ if there exists an exact sequence
in
$\mathbf{R}_{Qco}\textrm{-Mod}$,
$$0\to {\mathbf P}_n\to {\mathbf P}_{n-1}\to\ldots\to {\mathbf P}_0\to {\mathbf
M}_0\to 0, $$such that ${\mathbf P}_i\in {\mathcal V}$, for all $i=0,\ldots,n$.
\begin{theorem}\label{maint}
Let $X$ be a scheme having enough locally countably generated
vector bundles.
Let $\mathbf{F}$ be a flat quasi-coherent sheaf on $X$. Then
$\mathbf{F}=\varinjlim \mathbf{F}_i$, where
$\mathbf{F}_i$ is locally countably generated and flat with ${\mathcal
V}\mathrm{dim}\,{\mathbf F}_i\leq 1$.
\end{theorem}
\begin{proof}
Given a flat quasi-coherent sheaf
$\mathbf{F}$ we can find a short exact sequence
$$0\rightarrow \mathbf{M} \hookrightarrow \bigoplus_{j\in J}
\mathbf{P}_j \rightarrow \mathbf{F}\rightarrow 0,$$ where
$\mathbf{M}$ is locally flat Mittag-Leffler. By Theorem \ref{mlqc},
$\mathbf{M}=\bigcup_{i\in I} \mathbf{M}_i$, $\mathbf{M}_i$ is a
locally countably generated vector bundle for each $i\in I$. By the
same argument above, we are able to complete commutatively the following diagram
$$\xymatrix{0 \ar[r]& \mathbf{M}_i \ar@{^{(}->}[r]\ar@{^{(}->}[d] &
\bigoplus_{j\in J'}\mathbf{P}_j\ar@{^{(}->}[d] \ar[r] &
\mathbf{F}_{i,J'}\ar[r]\ar[d]& 0\\   
0\ar[r]&\mathbf{M}\ar@{^{(}->}[r]&\bigoplus_{j\in J}
\mathbf{P}_j\ar[r] & \mathbf{F} \ar[r] & 0 },$$ where
$J'\subseteq J$ is such that $\bigoplus_{j\in J'}\mathbf{P}_j$ is
locally countably generated and $\varinjlim \bigoplus_{j\in
J'}\mathbf{P}_j= \bigoplus_{j\in J}\mathbf{P}_j$ and $\mathbf
{F}_{i,J'}$ is locally countably generated and flat for each
$(i,J')$. Since $\mathbf{R}_{Qco}\textrm{-Mod}$ is a Grothendieck category, 
direct limits are exact, therefore $\varinjlim
\mathbf{F_{i,J'}} = \mathbf{F}$. Finally, since both $\mathbf{M}_i$ and
$\bigoplus_{j\in J'}\mathbf{P}_j$ are locally countably generated vector
bundles, it follows that $\mathcal{V}\mathrm{dim}\mathbf{F}_{i,J'}\leq 1$.

\end{proof}
In case $R$ is commutative noetherian and $X\subseteq \mathbb{P}_R^n$ is a
closed subscheme, we can replace in Theorem \ref{maint} the dimension with
respect to the class $\mathcal{V}$ by the projective dimension.
Recall that,
given a quasi--coherent sheaf $\mathbf{M}$, we say
that $\mathrm{projdim}\,\mathbf{M}\leq n$ if
$\mathrm{Ext}^i(\mathbf{M},-)=0$ for $i\geq n+1$. Then $\qco(X)$ has a family
of generators of projective dimension $\leq n$ (see \cite[pg. 538]{EnEsGa}). 
The generators are provided from the family of $\os(k)$,
$k\in\mathbb{Z}$, for ${\mathbb{P}^n}(R)$. These give the family
$\{i^*(\os(k)):\ k\in \mathbb{Z}\}$, where $i:X\hookrightarrow
{\mathbb{P}^n}(A)$ (see \cite[p.\ 120]{hartshorne} for notation and
terminology) we will let $\os(k)$ denote $i^*(\os(k))$. Then
\cite[Corollary 3.10]{EnEsGa}, shows that
$projdim\ \os(k)\leq n$ for all $k\in \mathbb{Z}$. Now using Serre's theorem
(see for example \cite[Corollary II.5.18]{hartshorne})
and that every quasi-coherent sheaf on $X$ is the
filtered union of coherent subsheafs, we get that
$\oplus_{l\in \Z} \os(k)$ is a generator for $\qco(X)$ of finite projective
dimension $\leq n$ by the previous.

In this case Theorem \ref{maint} specializes as follows:
\begin{corollary}\label{cons}
Let $R$ be a commutative noetherian and $X\subseteq \mathbb{P}^n(R)$ a closed
subscheme. Let $\mathbf{F}$ be a flat quasi--coherent sheaf on $X$. Then
$\mathbf{F}=\varinjlim \mathbf{F}_i$, where
$\mathbf{F}_i$ is locally countably generated and flat with
$\mathrm{projdim}\,{\mathbf F}_i\leq n+1$.
\end{corollary}

\begin{proof}
By the previous comments we can replace $\bigoplus_{j\in J}
\mathbf{P}_j $ by $\bigoplus_{j\in J}
\os (k_j)^{m_j}$ in the proof of Theorem \ref{maint}. Then from the proof
we get a short exact sequence $$0\to \mathbf{M}_i\to \bigoplus_{j\in J'}
\os (k_j)^{m'_j}\to \mathbf{F}_{i,J'}\to 0.$$ Now, for any quasi-coherent sheaf
$\mathbf{N}$, we have an exact sequence
$$\cdots \rightarrow \Ext^{l+1}(\mathbf{M_i},\mathbf{N})\rightarrow
\Ext^{l+2}(\mathbf{F}_{i,J'},\mathbf{N})\rightarrow
\Ext^{l+2}(\bigoplus_{j\in J'}\os(n_j)^{m'_j},\mathbf{N})\rightarrow
\cdots,$$ for each $l \geq 0$. But, since $\bigoplus_{j\in
J'}\os(n_j)^{m'_j}$ and $\mathbf{M}_i$ are locally projective, their
projective dimensions are $\leq n$ (cf. \cite[Corollary 3.10]{EnEsGa}). So we
get
$\Ext^{s+2}(\mathbf{F}_{i,J'},\mathbf{N})=0$, for each $s\geq n$.
That is, $\mathrm{projdim}(\mathbf{F}_{i,J'})\leq n+1$.
\end{proof}

\bigskip\par\noindent
{\bf Remarks:}
\begin{enumerate} 
\item According to \cite[Proposition 2.3]{hove} a more general
version of Corollary \ref{cons} holds on $\qco(X)$, for $X$ a noetherian
scheme with enough locally frees which is also separated or finite-dimensional.
\item Theorem \ref{maint} and Corollary \ref{cons} are also valid in case
$X=\mathbb{P}^1_R$ for $R$ commutative (need not be noetherian). This is
because the family $\{\mathcal{O}(k):\ k\in \mathbb{Z}\}$ is also a family of
generators of finite projective dimension $\leq 2$ in this case (see
\cite[Proposition 3.4]{flat}).
\end{enumerate}

\end{document}